\documentclass[final,nomarks]{dmtcs-episciences}

\usepackage[utf8]{inputenc}
\usepackage{subfigure}
\usepackage[round]{natbib}

\usepackage[dvipsnames]{xcolor}
\usepackage{graphicx,tikz,amssymb,mathtools,tabularx}
\usetikzlibrary{arrows,calc}
\usepackage[justification=centering]{caption}

\newcommand{\bob}{Builder}
\newcommand{\alice}{Algorithm}
\DeclareMathOperator{\ff}{First-Fit}

\DeclareMathOperator{\bd}{\mathbf{bd}}
\DeclareMathOperator{\ed}{\mathbf{ed}}
\DeclareMathOperator{\abcax}{\mathbf{abcax}}
\DeclareMathOperator{\aab}{\mathbf{aab}}
\DeclareMathOperator{\bab}{\mathbf{bab}}
\DeclareMathOperator{\abcde}{\mathbf{abcde}}
\DeclareMathOperator{\game}{\mathbf{gg}}
\DeclareMathOperator{\abcab}{\mathbf{abcab}}
\DeclareMathOperator{\abcac}{\mathbf{abcac}}
\DeclareMathOperator{\abcad}{\mathbf{abcad}}

\newtheorem{theorem}{Theorem}[section]
\newtheorem{lemma}[theorem]{Lemma}
\newtheorem{prop}[theorem]{Proposition}

\author[Israel R. Curbelo and Hannah R. Malko]{Israel R. Curbelo
  \and Hannah R. Malko}
\title[On the on-line coloring of unit interval graphs with proper interval representation]{On the on-line coloring of unit interval graphs with proper interval representation}

\affiliation{
  Department of Mathematical Sciences, Kean University, Union, NJ, USA}

\begin{document}
\publicationdata{vol. 27:2}{2025}{4}{10.46298/dmtcs.14088}{2024-08-21; 2024-08-21; 2025-02-10}{2025-02-13}

\maketitle

\begin{abstract}
We define the problem as a two-player game between \alice\ and \bob. The game is played in rounds. Each round, \bob\ presents an interval that is neither contained in nor contains any previously presented interval. \alice\ immediately and irrevocably assigns the interval a color that has not been assigned to any interval intersecting it. The set of intervals form an interval representation for a unit interval graph and the colors form a proper coloring of that graph. For every positive integer $\omega$, we define the value $R(\omega)$ as the maximum number of colors for which \bob\ has a strategy that forces \alice\ to use $R(\omega)$ colors with the restriction that the unit interval graph constructed cannot contain a clique of size $\omega+1$. In 1981, Chrobak and \'{S}lusarek showed that $R(\omega)\leq2\omega -1$. In 2005, Epstein and Levy showed that $R(\omega)\geq\lfloor{3\omega/2\rfloor}$. This problem remained unsolved for $\omega\geq 3$. In 2023, Bir\'o and Curbelo showed that $R(3)=5$. In this paper, we show that $R(4)=7$
\end{abstract}

\section{Introduction}

An on-line graph coloring algorithm $A$ receives a graph $G$ in the order of its vertices $v_1,\ldots,v_n$ and constructs an on-line coloring. This means that the color assigned to a vertex $v_i$ depends solely on the subgraph induced by the vertices $\{v_1,\ldots,v_i\}$ and on the colors assigned to them. The simplest on-line coloring algorithm is the greedy algorithm, $\ff$, which traverses the list of vertices given in the order they are presented, and assigns each vertex the minimal color not assigned to its neighbors that appear before it in the list of vertices. 

The performance of an algorithm $A$ on a graph $G$ is measured with respect to the chromatic number $\chi(G)$, which is equivalent to the number of colors needed by an optimal off-line algorithm. Let $\chi_A(G)$ be the maximum number of colors needed by algorithm $A$ over all orderings of the vertices of $G$. An on-line coloring algorithm $A$ is \emph{strictly $\rho$-competitive} if for every graph $G$, $\chi_A(G)\leq\rho\cdot\chi(G)$. The \emph{strict competitive ratio} of algorithm $A$ is $\inf\{\,\rho\mid\text{$A$ is strictly $\rho$-competitive}\,\}$. In general, $\ff$ does not have a bounded competitive ratio, and there is no on-line algorithm with a constant competitive ratio. The best known algorithm by \cite{hal-97,hal-sze-94} uses $O(\chi\cdot n/\log n)$ colors for $\chi$-colorable graphs of size $n$, and no on-line coloring algorithm can be $o(n/\log^2 n)$-competitive. For this reason, it is common to restrict the class of graphs. An important class of intersection graphs that has been extensively studied in this context, with applications in scheduling, is the class of interval graphs. 

A graph $G=(V,E)$ is an \emph{interval graph} if there is a function $f$ which assigns to each vertex $v\in V$ a closed interval $f(v)$ on the real line so that for all $u,v\in G$, $u$ and $v$ are adjacent if and only if $f(u)\cap f(v)\neq\emptyset$. We refer to $f$ an \emph{interval representation} of $G$. When no interval is properly contained in any other interval, we refer to $G$ as a \emph{proper interval graph} and $f$ as a \emph{proper interval representation} of $G$. When every interval has unit length, we refer to $G$ as a \emph{unit interval graph} and $f$ as a \emph{unit interval representation} of $G$. Since every unit interval representation is a proper interval representation, any unit interval graph is a proper interval graph. Moreover, it can be shown that an interval graph $G$ is a unit interval graph if and only if $G$ is a proper interval graph. 

The on-line coloring problem for the class of interval graphs was solved by \cite{kie-tro-81} when they constructed an on-line algorithm with competitive ratio $3$, and proved that no on-line algorithm could have a better competitive ratio. \cite{chr-slu-88} proved the same result, but instead of the verticies of the graph being presented as points, the graph was presented with its interval representation. Naturally, this leads us to question if knowledge of the interval representation makes a difference in the problem. Chrobak and \'Slusarek also considered the problem restricted to unit interval graphs with known representation. They showed that $\ff$ has a competitive ratio of $2$. Since $\ff$ does not use information from the representation, this holds true for both cases. However, the question of whether a better online algorithm exists remained unanswered. History would have us believe that there should be a more efficient algorithm. The most closely related example would be the performance of $\ff$ on the class of interval graphs. \cite{woo-74} was the first to consider this problem. After extensive research on this problem as seen in \cite{kie-88,kie-qin-95,kie-tro-81,chr-slu-88,pem-ram-04,nar-sub-08,kie-smi-16}, it was found that the competitive ratio of the $\ff$ algorithm for interval graphs is between 5 and 8. This is worse than the optimal algorithm developed by \cite{kie-tro-81}, which has a competitive ratio of 3.

\cite{eps-lev-05} presented a strategy for constructing a unit/proper interval graph with representation that would force any on-line algorithm to use $3k$ colors on a graph with clique size $k$. Hence, they showed that there is no on-line algorithm with competitive ratio less than $\frac{3}{2}$. In 2008, Micek studied this problem in terms of on-line chain partitioning and showed that there is no algorithm better than $\ff$ if the graph is presented without representation. (We refer the reader to the survey paper by \cite{survey} for a comprehensive overview of these problems in terms of on-line chain partitioning and for a proof of this result.) Specifically, there is no on-line algorithm with competitive ratio better than $2$. Micek also claimed that the strategy by \cite{eps-lev-05} could be improved if a proper interval representation was used.  
\cite{bir-cur-23} showed that the algorithm could also be improved for unit intervals. This provided a tight bound of $5$ when the chromatic number is restricted to at most $3$ and showed that there is no on-line algorithm with competitive ratio less than $\frac{5}{3}$ with unit or proper representation. 
It remains unknown whether knowledge of the representation or if the choice of representation makes a difference to the result of the problem.

In this paper, we present a strategy which forces any on-line algorithm to use $7$ colors on a unit interval graph with chromatic number at most $4$ and known proper interval representation and prove the following.
\begin{theorem}\label{the:main}
    There is no on-line algorithm with strict competitive ratio less than $\frac{7}{4}$ for the on-line coloring problem restricted to unit interval graphs with known proper interval representation.
\end{theorem}

\section{Preliminaries}

We define the problem as a two-player game between \bob\ and \alice. The game is played in rounds. Each round, \bob\ presents an interval so that any real number contained in that interval is contained in at most three other intervals. \alice, immediately and irrevocably, assigns the interval a color from the set $\mathcal{G}=\{a, b, c, d, e, f, g\}$ with the restriction that any two intersecting intervals cannot be assigned the same color. Anytime that \alice\ uses a new color, we may assume that it is the first unused color alphabetically in $\mathcal{G}$. In this paper, we play the role of \bob, attempting to force \alice\ to use the color $g$. We restrict the domain of the game so that all intervals presented by \bob\ must be properly contained in the interval $[0,1]$. At the end of round $i$, \bob\ may further restrict the domain by choosing two real numbers $l_i$ and $r_i$ so that $l_{i-1}\leq l_i<r_i\leq r_{i-1}$ and any interval presented after round $i$ must be properly contained in $[l_i,r_i]$. We refer to $l_i$ and $r_i$ as \emph{walls}. Note that this restriction makes no difference to the outcome of the game.

\subsection{State representation}
Consider a sequence of intervals $x_1, \ldots, x_n$ presented by \bob\ and a sequence of corresponding colors $y_1, \ldots, y_n$ assigned by \alice. The set of interval-color pairs  
\[
S = \{ (x_i,y_i) \mid i \in \{1, \ldots, n\} \}
\]
defines a \emph{state}.

Let $S$ be a state as defined above. We define its \emph{state representation matrix} by 
\[
    f(S)=\begin{bmatrix}
        s_1&s_2&\ldots&s_{2n}\\
        z_1&z_2&\ldots&z_{2n}\\
    \end{bmatrix}
    %f(S)=(s_1,z_1),\ldots,(s_{2n},z_{2n})
\]
where the sequences $s_1,\ldots,s_{2n}$ and $z_1,\ldots,z_{2n}$ are obtained as follows. First, let $r_1,\ldots,r_{2n}$ be obtained by taking all $2n$ endpoints of $x_1,\ldots,x_n$ and ordering them in ascending order. Next, for $i\in\{1,\ldots,2n\}$, $s_i=0$ if $r_i$ was a left endpoint and $s_i=1$ if $r_i$ was a right endpoint. Finally, $z_i$ denotes the color in $\mathcal{G}$ that was assigned to the interval which $r_i$ belonged to. 

We consider two states equivalent if they have the same state representation matrix. Furthermore, we consider a state $T$ equivalent to $S$ if and only if there exists a permutation $\sigma:\mathcal{G}\rightarrow\mathcal{G}$ so that 
\[
    f(T)=\begin{bmatrix}
        s_1&s_2&\ldots&s_{2n}\\
        \sigma z_1&\sigma z_2&\ldots&\sigma z_{2n}\\
    \end{bmatrix}.
\]

The following statement allows us to generalize and represent states visually throughout the paper.

\begin{prop}\label{prop12}
    Let $S_1$ and $S_2$ be states. If $S_1$ is equivalent to $S_2$ and there is a winning strategy for \bob\ starting at state $S_1$, then there is a winning strategy for \bob\ starting at state $S_2$.
\end{prop}

\begin{figure}
    \centering
    \begin{tabular}{|c|}
    \hline$\bd$\\
    \hline\resizebox{3.4in}{!}{\begin{tikzpicture}

\draw[-|, thick, Red] (-1,0) -- (0,0) node[below, black]{$ a$} -- (2,0);
\draw[-|, thick, OliveGreen] (-1,1) -- (1,1) node[below, black]{$b$} -- (3,1);
\draw[|-, thick, Red] (6,0) -- (9,0) node[below, black]{$a$} -- (10,0);
\draw[|-, thick, OliveGreen] (7,1) -- (9,1) node[below, black]{$b$} -- (10,1);
\draw[|-, thick, Orange] (8,2) -- (9,2) node[below, black]{$d$} -- (10,2);
\draw[|-|, thick, Blue] (2.5,2) -- (4.5,2) node[below, black]{$c$} -- (6.5,2);

\draw[dashed] (-1,-1) -- (-1,3);
\draw[dashed] (10,-1) -- (10,3);

\iffalse
\draw[-|] (-1,0) -- (0,0) node[below]{$a$} -- (2,0);
\draw[-|] (-1,0.5) -- (1,0.5) node[below]{$b$} -- (3,0.5);
\draw[|-] (6,0) -- (9,0) node[below]{$a$} -- (10,0);
\draw[|-] (7,0.5) -- (9,0.5) node[below]{$b$} -- (10,0.5);
\draw[|-] (8,1) -- (9,1) node[below]{$d$} -- (10,1);
\draw[|-|] (2.5,1) -- (4.5,1) node[below]{$c$} -- (6.5,1);

\draw[dashed] (-1,-0.5) -- (-1,1.5);
\draw[dashed] (10,-0.5) -- (10,1.5);
\fi

%\draw[fill=black] (-1.5,-1) rectangle (-0.5,3);
%\draw[fill=black] (9.5,-1) rectangle (10.5,3);
%\draw (4.5,3) node[above] {\texttt{bd}};

\end{tikzpicture}}\\
    \hline\hline$\ed$\\
    \hline\resizebox{3.4in}{!}{\begin{tikzpicture}

\draw[-|, thick, Red] (-1,0) -- (0,0) node[below, black]{$a$} -- (2,0);
\draw[-|, thick, OliveGreen] (-1,1) -- (1,1) node[below, black]{$b$} -- (3,1);
\draw[|-, thick, Red] (6,0) -- (9,0) node[below, black]{$a$} -- (10,0);
\draw[|-, thick, Purple] (7,1) -- (9,1) node[below, black]{$e$} -- (10,1);
\draw[|-, thick, Orange] (8,2) -- (9,2) node[below, black]{$d$} -- (10,2);
\draw[|-|, thick, Blue] (2.5,2) -- (4.5,2) node[below, black]{$c$} -- (6.5,2);

\draw[dashed] (-1,-1) -- (-1,3);
\draw[dashed] (10,-1) -- (10,3);

\end{tikzpicture}}\\
    \hline\hline$\abcax$\\
    \hline\resizebox{3.4in}{!}{\begin{tikzpicture}

\draw[-|, thick, Red] (0,0) -- (2,0) node[below, black]{$a$} -- (4,0);
\draw[|-|, thick, Red] (8,0) -- (10,0) node[below, black]{$a$} -- (12,0);
\draw[|-|, thick, OliveGreen] (3,1) -- (5,1) node[below, black]{$b$} -- (7,1);
\draw[|-|] (11,1) -- (13,1) node[below]{$x\in\{b,c,d\}$} -- (15,1);
\draw[|-|, thick, Blue] (5,2) -- (7,2) node[below, black]{$c$} -- (9,2);

\draw[dashed] (0,-1) -- (0,3);
\draw[dashed] (15.5,-1) -- (15.5,3);
    
\end{tikzpicture}}\\
    \hline\hline$\aab$\\
    \hline\resizebox{3.4in}{!}{\begin{tikzpicture}

\draw[-|, thick, Red] (0,0) -- (2,0) node[below, black]{$a$} -- (4,0);
\draw[|-|, thick, Red] (8,0) -- (10,0) node[below, black]{$a$} -- (12,0);
\draw[|-|, thick, OliveGreen] (11,1) -- (13,1) node[below, black]{$b$} -- (15,1);

\draw[dashed] (0,-1) -- (0,2);
\draw[dashed] (15.5,-1) -- (15.5,2);

\end{tikzpicture}}\\
    \hline\hline$\bab$\\
    \hline\resizebox{3.4in}{!}{\begin{tikzpicture}

\draw[-|, thick, OliveGreen] (0,0) -- (2,0) node[below, black]{$b$} -- (4,0);
\draw[|-|, thick, Red] (8,0) -- (10,0) node[below, black]{$a$} -- (12,0);
\draw[|-|, thick, OliveGreen] (11,1) -- (13,1) node[below, black]{$b$} -- (15,1);

\draw[dashed] (0,-1) -- (0,2);
\draw[dashed] (15.5,-1) -- (15.5,2);

\end{tikzpicture}}\\
    \hline\hline$\abcde$\\
    \hline\resizebox{3.4in}{!}{\begin{tikzpicture}

\draw[-|, thick, OliveGreen] (0,0) -- (2,0) node[below, black]{$b$} -- (4,0);
\draw[|-|, thick, Red] (8,0) -- (10,0) node[below, black]{$a$} -- (12,0);
\draw[|-|, thick, OliveGreen] (11,1) -- (13,1) node[below, black]{$b$} -- (15,1);

\draw[dashed] (0,-1) -- (0,2);
\draw[dashed] (15.5,-1) -- (15.5,2);

\end{tikzpicture}}\\
    \hline
    \end{tabular}
    \caption{States}
    \label{fig:table}
\end{figure}

Given a state $S$, let $S^*$ denote the \emph{dual} of $S$ defined as a state constructed by reversing the order of the columns in the state representation matrix of $S$. The following statement allows us switch between a state and its dual in our strategy.

\begin{prop}\label{prop:reverse}
    Let $S$ be a state. If there is a winning strategy for \bob\ starting at state $S$, then there is a winning strategy for \bob\ starting at state $S^*$.  
\end{prop}
 
To further generalize states to includes walls, we simply ignore endpoints outside of the walls. For example, we say that a state $S$ is a $\bd$ state (shown in Figure \ref{fig:table}) if there exists a permutation $\sigma:\mathcal{G}\rightarrow\mathcal{G}$ such that 
\[
\begin{bmatrix}
    1&0&1&0&1&0&0\\
    \sigma a&\sigma c&\sigma b&\sigma a&\sigma c&\sigma b &\sigma d
\end{bmatrix}
\]
is a submatrix of contiguous columns of $f(S)$. Additionally, we say that a state $S$ is a $\bd^*$ state if there exists a permutation $\sigma:\mathcal{G}\rightarrow\mathcal{G}$ such that 
\[
\begin{bmatrix}
    0&0&1&0&1&0&1\\
    \sigma d&\sigma b&\sigma c&\sigma a&\sigma b&\sigma c &\sigma a
\end{bmatrix}
\] 
is a submatrix of contiguous columns of $f(S^*)$, that is, if $S^*$ is a $\bd$ state. For convenience, we define a $\game$ state as any state that contains all seven colors in $\mathcal{G}$. 

\subsection{Separation strategy}
We generalize the separation strategy used by \cite{eps-lev-05} and provide a proof of the following lemma for completeness.

\begin{lemma}\label{lem}
Let $[\alpha_1,\beta_1]$ and $[\alpha_2,\beta_2]$ be intervals with $\beta_1<\alpha_2$, let $Y$ be a set of colors, and let $k$ be a positive integer. Then for every on-line algorithm there is a strategy for \bob\ to present $k$ intervals so that they will form a set $\{[l_1,r_1],\ldots,[l_k,r_k]\}$ with $\alpha_1<l_1<\ldots<l_k<\beta_1<\alpha_2<r_1<\ldots<r_k<\beta_2$, and there is an integer $j\in\{0,\ldots,k\}$ such that for every $i\in\{1,\ldots,k\}$, \alice\ assigns the interval $[l_i,r_i]$ a color in $Y$ if and only if $i\leq j$.
\end{lemma}
\begin{proof}
Let $[\alpha_1,\beta_1]$ and $[\alpha_2,\beta_2]$ be intervals with $\beta_1<\alpha_2$, and let $Y$ be a set of colors. We argue by induction on the positive integer $k$. If $k=1$, then we present the interval $I=[(\alpha_1+\beta_1)/2,(\alpha_2+\beta_2)/2]$. If \alice\ assigns the interval $I$ a color in $Y$, then $j=1$. Otherwise, $j=0$.

Suppose $k>1$. By the induction hypothesis, there exists a strategy for \bob\ to present $k-1$ intervals so that they will form a set $\{[l'_1,r'_1],\ldots,[l'_{k-1},r'_{k-1}]\}$ with $\alpha_1<l'_1<\ldots<l'_{k-1}<\beta_1<\alpha_2<r'_1<\ldots<r'_{k-1}<\beta_2$ and a integer $j'\in\{0,\ldots,k-1\}$ such that for every $i\in\{1,\ldots,k-1\}$, \alice\ assigns the interval $[l'_i,r'_i]$ a color from $Y$ if and only if $i\leq j'$. If $j'=0$, we introduce the interval $I=[(\alpha_1+l'_1)/2, (\alpha_2+r'_1)/2]$ and define $[l_i,r_i]=[l'_{i-1},r'_{i-1}]$ for $i\in\{2,\ldots,k\}$ and $[l_1,r_1]=I$. If $j'=k-1$, we introduce the interval $I=[(l'_{k-1}+\beta_1)/2, (r'_{k-1}+\beta_2)/2]$ and define $[l_i,r_i]=[l'_i,r'_i]$ for $i\in\{1,\ldots,k-1\}$ and $[l_k,r_k]=I$. Otherwise, we introduce the interval $I=[(l'_{j'}+l'_{j'+1})/2, (r'_{j'}+r'_{j'+1})/2]$ and define $[l_i,r_i]=[l'_i,r'_i]$ for $i\in\{1,\ldots,j\}$, $[l_i,r_i]=[l'_{i-1},r'_{i-1}]$ for $i\in\{j+2,\ldots,k\}$ and $[l_{j+1},r_{j+1}]=I$. If \alice\ assigns the new interval $I$ a color in $Y$, then $j=j'+1$. Otherwise, $j=j'$.
\end{proof}

Lemma \ref{lem} implies that if $S$ is a state with state representation matrix 
\[
    f(S)=\begin{bmatrix}
        s_1&s_2&\ldots&s_{2n}\\
        z_1&z_2&\ldots&z_{2n}\\
    \end{bmatrix},
    %f(S)=(s_1,z_1),\ldots,(s_{2n},z_{2n})
\]
then for any positive integers $p$ and $q$ with $1\leq p<q\leq 2n$,
positive integer $k$ and set of colors $Y\subseteq\mathcal{G}$, 
we can introduce $k$ intervals so that the result is a state $S'$ with state representation matrix
\[
    f(S')=\begin{bmatrix}
        s_1\ldots s_{p-1}&0\ldots 0&s_{p}\ldots s_{q}&1\ldots 1&s_{q+1}\ldots s_{2n}\\
        z_1\ldots z_{p-1}&z'_1\ldots z'_k&z_{p}\ldots z_q&z'_1\ldots z'_k&z_{q+1}\ldots z_{2n}\\
    \end{bmatrix}
    %f(S)=(s_1,z_1),\ldots,(s_{2n},z_{2n})
\]
and there is an integer $j\in\{0,\ldots,k\}$ such that for every $i\in\{1,\ldots,k\}$, $z'_i\in Y$ if and only if $i\leq j$. Note that when $p=1$, the first column of $f(S')$ is $(0,z'_1)$, and when $q=2n$, the last column of $f(S')$ is $(1,z'_k)$.

We use this separation strategy throughout the proof of Theorem \ref{the:main} in the form of \emph{``we present $k$ intervals as shown in Figure $t$ while separating any interval assigned a color in $Y$ to the left''} or equivalently \emph{``we present $k$ intervals as shown in Figure $t$ while separating any interval assigned a color in $\mathcal{G}-Y$ to the right''} with $t$ denoting the Figure that defines $[\alpha_1,\beta_1]$ and $[\alpha_2,\beta_2]$ by the nearest endpoints to the left and to the right of the left and right endpoints of the first new interval, respectively. When $Y=\{y\}$ for some $y\in\mathcal{G}$, we simply state \emph{``assigned the color $y$''} instead of \emph{``assigned a color in $Y$''}.  

\subsection{Outline of proof}
\begin{figure}[h]
    \centering
    \begin{tikzpicture}
        \node[draw=black] (s) at (0,0) {$\varnothing$};
        \node[draw=black] (bab*) at ($(s)-(2,1)$) {$\bab^*$};
        \node[draw=black] (aab*) at ($(s)-(0,1)$) {$\aab^*$};
        \node[draw=black] (bab) at ($(s)-(2,2)$) {$\bab$};
        \node[draw=black] (aab) at ($(s)-(0,2)$) {$\aab$};
        \node[draw=black] (abcde) at ($(s)-(-2,1)$) {$\abcde$};
        \node[draw=black] (gg) at ($(abcde)-(0,4)$) {$\game$};
        \node[draw=black] (abcab) at ($(bab)-(2,1)$) {$\abcab$};
        \node[draw=black] (abcac) at ($(bab)-(0,1)$) {$\abcac$};
        \node[draw=black] (abcad) at ($(bab)-(-2,1)$) {$\abcad$};
        \node[draw=black] (bd) at ($(abcad)-(2,1)$) {$\bd$};
        \node[draw=black] (ed) at ($(abcad)-(0,1)$) {$\ed$};

        \path[->] (s) edge (bab*);
        \path[->] (s) edge (aab*);
        \path[<->] (bab*) edge (bab);
        \path[<->] (aab*) edge (aab);
        \path[->] (s) edge (abcde);
        \path[->] (abcde) edge (aab);
        %\path[->] (abcde) edge (bab);
        \path[->] (abcde) edge (gg);
        \path[->] (aab) edge (abcab);
        \path[->] (aab) edge (abcac);
        \path[->] (aab) edge (abcad);
        \path[->] (bab) edge (abcab);
        \path[->] (bab) edge (abcac);
        \path[->] (bab) edge (abcad);
        \path[->] (abcab) edge (bd);
        \path[->] (abcac) edge (bd);
        \path[->] (abcac) edge (ed);
        \path[->] (abcad) edge (bd);
        \path[->] (abcad) edge (ed);
        \path[->] (abcad) edge (gg);
        \path[->] (bd) edge (gg);
        \path[->] (ed) edge (gg);
        
    \end{tikzpicture}
    \caption{Outline of the proof of Theorem \ref{the:main}.}
    \label{fig:outline}
\end{figure}

In Section \ref{sec:proof}, we prove Theorem \ref{the:main} by showing that there is a strategy for \bob\ which forces a $\game$ state starting from each of the states in Figure \ref{fig:table} and then showing that we can always force at least one of those states or the dual of one of those states. Note that the third state $\abcax$ in Figure \ref{fig:table} contains an interval labeled with an arbitrary color $x\in\{b,c,d\}$. This is to represent three states $\abcab$, $\abcac$ and $\abcad$ where the interval is colored $b$, $c$ and $d$, respectively. That is, a state is an $\abcax$ state if and only if it is an $\abcab$ state, an $\abcac$ state or an $\abcad$ state. 

The structure of the proof is as follows. In Section \ref{sec:bd} and Section \ref{sec:ed}, we show that we can force a $\game$ state starting from any $\bd$ state or from any $\ed$ state, respectively. In Section \ref{sec:abcab}, Section \ref{sec:abcac} and Section \ref{sec:abcad}, we show that that we can force a $\game$ state starting from any $\abcax$ state. In particular, in Section \ref{sec:abcab}, we show that we can force a $\bd$ state starting from any $\abcab$ state, and in Section \ref{sec:abcac}, we show that we can force a $\bd$ or an $\ed$ state starting from any $\abcac$ state. In Section \ref{sec:abcad}, we show that we can force a $\game$ state starting from any $\abcad$. In Section \ref{sec:aborbab}, we show that we can force an $\abcax$ state starting from either any $\aab$ state or any $\bab$ state. In Section \ref{sec:abcde}, we show that we can force a $\game$ state starting from any $\abcde$ state. Finally, in Section \ref{sec:gg}, we show that we can always force either an $\aab^*$ state, a $\bab^*$ state or an $\abcde$ state. The complete outline is summarized in Figure \ref{fig:outline}.

\section{Proof of Main Theorem}
\label{sec:proof}

\subsection{$(\bd\rightarrow\game)$}
\label{sec:bd}

Assume we start in a $\bd$ state. Figure \ref{fig:bdproof} shows how to force a $\game$ state in two phases. In the first phase, we present two intervals as shown in Figure \ref{fig:bdproof} while separating any interval assigned the color $d$ to the left. Let $x$ be the color assigned to the interval on the left and $y$ be the color assigned to the interval on the right. In the second phase, there are two cases. That is, either \alice\ used the color $d$ in the first phase or \alice\ did not use the color $d$ in the first phase. If \alice\ used the color $d$ in the first phase, then $x=d$ and $y=e$. The bottom left frame of Figure \ref{fig:bdproof} shows how to force a $\game$ state in two moves. If neither interval was assigned the color $d$, then both intervals were assigned new colors. Without loss of generality, let us assume that $x=f$ and $y=e$. The bottom right frame of Figure \ref{fig:bdproof} shows how to force a $\game$ state in one move. 

\begin{figure}[h]
    \centering
    \begin{tikzpicture}
        \node[draw=black] (A) at (0,-1.2) {$\bd$};
        \node[draw=black] (B) at (0,-4) {\resizebox{\textwidth/2}{!}{\begin{tikzpicture}

\draw[-|, thick, Red] (-1,0) -- (0,0) node[below, black]{$a$} -- (2,0);
\draw[-|, thick, OliveGreen] (-1,1) -- (1,1) node[below, black]{$b$} -- (3,1);
\draw[|-, thick, Red] (6,0) -- (9,0) node[below, black]{$a$} -- (10,0);
\draw[|-, thick, OliveGreen] (7,1) -- (9,1) node[below, black]{$b$} -- (10,1);
\draw[|-, thick, Orange] (8,2) -- (9,2) node[below, black]{$d$} -- (10,2);
\draw[|-|, thick, Blue] (2.5,2) -- (4.5,2) node[below, black]{$c$} -- (6.5,2);
\draw[very thick, |-|] (0,3) -- (2,3) node[below, black]{$x$} -- (4,3);
\draw[very thick, |-|] (1,4) -- (3,4) node[below, black]{$y\notin\{d\}$} -- (5,4);

\draw[dashed] (-1,-1) -- (-1,5);
\draw[dashed] (10,-1) -- (10,5);

\end{tikzpicture}}};
        %\node[draw=black] (A1) left of (B) {1};
        \node[draw=black] (C) at (-3.2,-8.8) {\resizebox{2.2in}{!}{\begin{tikzpicture}

\draw[-|, thick, Red] (-1,0) -- (0,0) node[below, black]{$a$} -- (2,0);
\draw[-|, thick, OliveGreen] (-1,1) -- (1,1) node[below, black]{$b$} -- (3,1);
\draw[|-, thick, Red] (6,0) -- (9,0) node[below, black]{$a$} -- (10,0);
\draw[|-, thick, OliveGreen] (7,1) -- (9,1) node[below, black]{$b$} -- (10,1);
\draw[|-, thick, Orange] (8,2) -- (9,2) node[below, black]{$d$} -- (10,2);
\draw[|-|, thick, Blue] (2.5,2) -- (4.5,2) node[below, black]{$c$} -- (6.5,2);
\draw[|-|, thick, Orange] (0,3) -- (2,3) node[below, black]{$d$} -- (4,3);
\draw[|-|, thick, Purple] (1,4) -- (3,4) node[below, black]{$e$} -- (5,4);
\draw[|-|, very thick, Cyan] (3.5,5) -- (5.5,5) node[below, black]{$f$} -- (7.5,5);
\draw[|-|, very thick, Magenta] (4.5,6) -- (6.5,6) node[below, black]{$g$} -- (8.5,6);

\draw[dashed] (-1,-1) -- (-1,7);
\draw[dashed] (10,-1) -- (10,7);

\end{tikzpicture}}};
        \node[draw=black] (D) at (3.2,-8.8) {\resizebox{2.2in}{!}{\begin{tikzpicture}

\draw[-|, thick, Red] (-1,0) -- (0,0) node[below, black]{$a$} -- (2,0);
\draw[-|, thick, OliveGreen] (-1,1) -- (1,1) node[below, black]{$b$} -- (3,1);
\draw[|-, thick, Red] (6,0) -- (9,0) node[below, black]{$a$} -- (10,0);
\draw[|-, thick, OliveGreen] (7,1) -- (9,1) node[below, black]{$b$} -- (10,1);
\draw[|-, thick, Orange] (8,2) -- (9,2) node[below, black]{$d$} -- (10,2);
\draw[|-|, thick, Blue] (2.5,2) -- (4.5,2) node[below, black]{$c$} -- (6.5,2);
\draw[|-|, thick, Cyan] (0,3) -- (2,3) node[below, black]{$f$} -- (4,3);
\draw[|-|, thick, Purple] (1,4) -- (3,4) node[below, black]{$e$} -- (5,4);
\draw[very thick, |-|, Magenta] (3.5,5) -- (6,5) node[below, black]{$g$} -- (8.5,5);

\draw[dashed] (-1,-1) -- (-1,7);
\draw[dashed] (10,-1) -- (10,7);

\end{tikzpicture}}};
        \path [->] (A) edge (B);
        \path [->] (B) edge (C);
        \path [->] (B) edge (D);

    \end{tikzpicture}
    \caption{$(\bd\rightarrow\game)$}
    \label{fig:bdproof}
\end{figure}

\subsection{$(\ed\rightarrow\game)$}
\label{sec:ed}

Assume we start in an $\ed$ state. Figure \ref{fig:edproof} shows how to force a $\game$ state in two phases. In the first phase, we present two intervals as shown in Figure \ref{fig:edproof} while separating any interval assigned the color $e$ to the left. Let $y$ denote the color assigned to the right interval. Since $y\neq e$, it must be the case that $y\in\{d,f\}$, and hence, there are two cases in the second phase. If $y=d$, then the bottom left frame of Figure \ref{fig:bdproof} shows how to force a $\game$ state in two moves. If $y=f$, the bottom right frame of Figure \ref{fig:edproof} shows how to force a $\game$ state in one move.

\begin{figure}[h]
    \centering
    \begin{tikzpicture}
        \node[draw=black] (A) at (0,-1.2) {$\ed$};
        \node[draw=black] (B) at (0,-4) {\resizebox{\textwidth/2}{!}{\begin{tikzpicture}

\draw[-|, thick, Red] (-1,0) -- (0,0) node[below, black]{$a$} -- (2,0);
\draw[|-|, thick, OliveGreen] (1,1) -- (3,1) node[below, black]{$b$} -- (5,1);
\draw[|-, thick, Red] (6,0) -- (9,0) node[below, black]{$a$} -- (10,0);
\draw[|-, thick, Purple] (7,1) -- (9,1) node[below, black]{$e$} -- (10,1);
\draw[|-, thick, Orange] (8,2) -- (9,2) node[below, black]{$d$} -- (10,2);
\draw[|-|, thick, Blue] (2.5,2) -- (4.5,2) node[below, black]{$c$} -- (6.5,2);

\draw[very thick, |-|] (-0.9,3) -- (1.1,3) -- (3.1,3);
\draw[very thick, |-|] (0,4) -- (2,4) node[below, black]{$y\notin\{e\}$} -- (4,4);

\draw[dashed] (-1,-1) -- (-1,5);
\draw[dashed] (10,-1) -- (10,5);

\end{tikzpicture}}};
        \node[draw=black] (C) at (-3.2,-8.8) {\resizebox{2.2in}{!}{\begin{tikzpicture}

\draw[-|, thick, Red] (-1,0) -- (0,0) node[below, black]{$a$} -- (2,0);
\draw[|-|, thick, OliveGreen] (1,1) -- (3,1) node[below, black]{$b$} -- (5,1);
\draw[|-, thick, Red] (6,0) -- (9,0) node[below, black]{$a$} -- (10,0);
\draw[|-, thick, Purple] (7,1) -- (9,1) node[below, black]{$e$} -- (10,1);
\draw[|-, thick, Orange] (8,2) -- (9,2) node[below, black]{$d$} -- (10,2);
\draw[|-|, thick, Blue] (2.5,2) -- (4.5,2) node[below, black]{$c$} -- (6.5,2);
\draw[|-|, thick] (-0.9,3) -- (1.1,3) -- (3.1,3);
\draw[|-|, thick, Orange] (0,4) -- (2,4) node[below, black]{$d$} -- (4,4);
\draw[very thick, |-|, Cyan] (3.5,5) -- (5.5,5) node[below, black]{$f$} -- (7.5,5);
\draw[very thick, |-|, Magenta] (4.5,6) -- (6.5,6) node[below, black]{$g$} -- (8.5,6);

\draw[dashed] (-1,-1) -- (-1,7);
\draw[dashed] (10,-1) -- (10,7);

\end{tikzpicture}}};
        \node[draw=black] (D) at (3.2,-8.8) {\resizebox{2.2in}{!}{\begin{tikzpicture}

\draw[-|, thick, Red] (-1,0) -- (0,0) node[below, black]{$a$} -- (2,0);
\draw[|-|, thick, OliveGreen] (1,1) -- (3,1) node[below, black]{$b$} -- (5,1);
\draw[|-, thick, Red] (6,0) -- (9,0) node[below, black]{$a$} -- (10,0);
\draw[|-, thick, Purple] (7,1) -- (9,1) node[below, black]{$e$} -- (10,1);
\draw[|-, thick, Orange] (8,2) -- (9,2) node[below, black]{$d$} -- (10,2);
\draw[|-|, thick, Blue] (2.5,2) -- (4.5,2) node[below, black]{$c$} -- (6.5,2);
\draw[|-|, thick] (-0.9,3) -- (1.1,3) -- (3.1,3);
\draw[|-|, thick, Cyan] (0,4) -- (2,4) node[below, black]{$f$} -- (4,4);
\draw[very thick, |-|, Magenta] (3.5,5) -- (6.5,5) node[below, black]{$g$} -- (8.5,5);

\draw[dashed] (-1,-1) -- (-1,7);
\draw[dashed] (10,-1) -- (10,7);

\end{tikzpicture}}};
        \path [->] (A) edge (B);
        \path [->] (B) edge (C);
        \path [->] (B) edge (D);
    \end{tikzpicture}
    \caption{$\ed\rightarrow\game$}
    \label{fig:edproof}
\end{figure}

\subsection{$(\abcab\rightarrow\bd)$}
\label{sec:abcab}

Assume we start in an $\abcab$ state. Figure \ref{fig:abcabproof} shows how to force a $\bd$ state in at most two moves. We present two intervals and update the right wall as shown in Figure \ref{fig:abcabproof} while separating any interval that was assigned the color $c$ to the right. Let $x$ denote the color that \alice\ assigned to the left interval. Since neither interval could have been assigned the colors $a$ or $b$, the left interval must have been colored with a new color. Without loss of generality, we may assume that $x=d$, resulting in a $\bd$ state. 

\subsection{$(\abcac\rightarrow\bd\lor\ed)$}
\label{sec:abcac}

Assume we start in an $\abcac$ state. Figure \ref{fig:abcacproof} shows how we can force either a $\bd$ state or an $\ed$ state in at most two moves. We present two intervals and update the right wall as shown in Figure \ref{fig:abcacproof} while separating any interval assigned the color $b$ to the left. Let $x$ denote the color assigned to the left interval and $y$ denote the color assigned to the right interval. If \alice\ assigned the color $b$ to either of the two intervals, then $x=b$ and $y=d$ which results in a $\bd$ state. If \alice\ did not assign the color $b$ to either interval, then both intervals were assigned a new color. Without loss of generality, we may assume $x=e$ and $y=d$ resulting in an $\ed$ state. 

\begin{figure}[h]
    \centering
    \begin{tikzpicture}
        \node[draw=black] (abcab) at (0,-0.5) {$\abcab$};
        \node[draw=black] (abcabd) at (0,-3) {\resizebox{4in}{!}{\begin{tikzpicture}

\draw[-|, thick, Red] (0,0) -- (2,0) node[below, black]{$a$} -- (4,0);
\draw[|-|, thick, Red] (9,0) -- (11,0) node[below, black]{$a$} -- (13,0);
\draw[|-|, thick, OliveGreen] (3,1) -- (5,1) node[below, black]{$b$} -- (7,1);
\draw[|-|, thick, OliveGreen] (10,1) -- (12,1) node[below, black]{$b$} -- (14,1);
\draw[|-|, thick, Blue] (5.5,2) -- (7.5,2) node[below, black]{$c$} -- (9.5,2);
\draw[very thick, |-|, Orange] (11,2) -- (13,2) node[below, black]{$x\in\{d,e\}$} -- (15,2);
\draw[very thick, |-|] (12,3) -- (14,3) -- (16,3);

\draw[dashed] (0,-1) -- (0,4);
\draw[very thick, dashed] (11.5,-1) -- (11.5,4);
    
\end{tikzpicture}}};
        \node (bd) at (-1,-5.5) {\scriptsize$x=d$};
        \node (bd2) at (1,-5.5) {\scriptsize$x=e$};
        \node[draw=black] (b) at (0,-6.5) {$\bd$};

        \draw[->] (abcab) edge (abcabd);
        \draw[-] (abcabd) edge (bd);
        \draw[-] (abcabd) edge (bd2);
        \draw[->] (bd) edge (b);
        \draw[->] (bd2) edge (b);

    \end{tikzpicture}
    \caption{$\abcab\rightarrow\bd$}
    \label{fig:abcabproof}
\end{figure}

\begin{figure}[h!]
    \centering
    \begin{tikzpicture}
        \node[draw=black] (abcac) at (0,-0.5) {$\abcac$};
        \node[draw=black] (abcac-xy) at (0,-3) {\resizebox{4in}{!}{\begin{tikzpicture}

\draw[-|, thick, Red] (0,0) -- (2,0) node[below, black]{$a$} -- (4,0);
\draw[|-|, thick, Red] (8,0) -- (10,0) node[below, black]{$a$} -- (12,0);
\draw[|-|, thick, OliveGreen] (3,1) -- (5,1) node[below, black]{$b$} -- (7,1);
\draw[|-|, thick, Blue] (11,1) -- (13,1) node[below, black]{$c$} -- (15,1);
\draw[|-|, thick, Blue] (4.5,2) -- (6.5,2) node[below, black]{$c$} -- (8.5,2);
\draw[very thick, |-|] (9,2) -- (11,2) node[below, black]{$x$} -- (13,2);
\draw[very thick, |-|] (10,3) -- (12,3) node[below, black]{$y\notin\{b\}$} -- (14,3);

\draw[dashed] (0,-1) -- (0,4);
\draw[very thick, dashed] (10.5,-1) -- (10.5,4);
    
\end{tikzpicture}}};
        \node (abcac-bd) at (-4,-6) {\scriptsize$x=b,y=d$};
        %\node (abcac-ed) at (0,-6) {\scriptsize$x=e,y=d$};
        \node (abcac-de) at (4,-6) {\scriptsize$x=e,y=d$};
        \node[draw=black] (bd) at (-4,-7) {$\bd$};
        \node[draw=black] (ed) at (4,-7) {$\ed$};

        \draw[->] (abcac) edge (abcac-xy);
        \draw[-] (abcac-xy) edge (abcac-bd);
        %\draw[-] (abcac-xy) edge (abcac-ed);
        \draw[-] (abcac-xy) edge (abcac-de);
        %\draw[->] (abcac-ed) edge (abcac-de);
        \draw[->] (abcac-bd) edge (bd);
        %\draw[->] (abcac-ed) edge (ed);
        \draw[->] (abcac-de) edge (ed);
    \end{tikzpicture}
    \caption{$\abcac\rightarrow\bd\lor\ed$}
    \label{fig:abcacproof}
\end{figure}

\subsection{$\abcad\rightarrow\game$}\label{sect:abcad}
\label{sec:abcad}

\begin{figure}
    \centering
    \begin{tikzpicture}
        \node[draw=black] (abcad) at (0,2) {$\abcad$};
        \node[draw=black] (abcadx) at (0,0) {\resizebox{4.5in}{!}{\begin{tikzpicture}

\draw[-|, thick, Red] (0,0) -- (2,0) node[below, black]{$a$} -- (4,0);
\draw[|-|, thick, Red] (8.5,0) -- (10.5,0) node[below, black]{$a$} -- (12.5,0);
\draw[|-|, thick, OliveGreen] (3,1) -- (5,1) node[below, black]{$b$} -- (7,1);
\draw[|-|, thick, Orange] (11,1) -- (13,1) node[below, black]{$d$} -- (15,1);
\draw[|-|, thick, Blue] (5,2) -- (7,2) node[below, black]{$c$} -- (9,2);
\draw[very thick, |-|] (12,2) -- (14,2) node[below, black]{$x\in\{b,c,e\}$} -- (16,2);

\draw[dashed] (0,-1) -- (0,3);
\draw[dashed] (17,-1) -- (17,3);
    
\end{tikzpicture}}};
        \node (abcade) at (-4,-2) {$x=e$};
        \node[draw=black] (ed1) at ($(abcade)+(0,-1.5)$) {$\ed$};
        \node (abcadc) at (4,-2) {$x=c$};
        \node[draw=black] (blank1) at ($(abcadc) + (0,-0.75)$) {}; 
        \node[draw=black] (ed2) at ($(blank1)+(-0.5,-0.75)$) {$\ed$};
        \node[draw=black] (bd2) at ($(blank1)+(0.5,-0.75)$) {$\bd$};
        \node (abcadb) at (0,-2) {$x=b$};
        \node[draw=black] (abcadbx) at ($(abcadb)+(0,-4)$) {\resizebox{4.5in}{!}{\begin{tikzpicture}

\draw[-|, thick, Red] (0,0) -- (2,0) node[below, black]{$a$} -- (4,0);
\draw[|-|, thick, Red] (8.5,0) -- (10.5,0) node[below, black]{$a$} -- (12.5,0);
\draw[|-|, thick, OliveGreen] (3,1) -- (5,1) node[below, black]{$b$} -- (7,1);
\draw[|-|, thick, Orange] (11,1) -- (13,1) node[below, black]{$d$} -- (15,1);
\draw[|-|, thick, Blue] (5,2) -- (7,2) node[below, black]{$c$} -- (9,2);
\draw[|-|, thick, OliveGreen] (12,2) -- (14,2) node[below, black]{$b$} -- (16,2);
\draw[very thick, |-|] (14,3) -- (16,3) -- (18,3);
\draw[very thick, |-|] (13,4) -- (15,4) node[below, black]{$y\in\{c,e\}$} -- (17,4);

\draw[dashed] (0,-1) -- (0,5);
\draw[dashed] (18.5,-1) -- (18.5,5);
    
\end{tikzpicture}}};
        \node (abcadbc) at (-4,-8.5) {$y=c$};
        \node[draw=black] (blank2) at ($(abcadbc) + (0,-0.75)$) {};
        \node[draw=black] (ed3) at ($(blank2)+(0,-0.75)$) {$\ed$};
        \node (abcadbe) at (0,-8.5) {$y=e$};
        \node[draw=black] (abcadbex) at ($(abcadbe)+(0,-4)$) {\resizebox{4.5in}{!}{\begin{tikzpicture}

\draw[-|, thick, Red] (0,0) -- (2,0) node[below, black]{$a$} -- (4,0);
\draw[|-|, thick, Red] (8.5,0) -- (10.5,0) node[below, black]{$a$} -- (12.5,0);
\draw[|-|, thick, OliveGreen] (3,1) -- (5,1) node[below, black]{$b$} -- (7,1);
\draw[|-|, thick, Orange] (11,1) -- (13,1) node[below, black]{$d$} -- (15,1);
\draw[|-|, thick, Blue] (5,2) -- (7,2) node[below, black]{$c$} -- (9,2);
\draw[|-|, thick, OliveGreen] (12,2) -- (14,2) node[below, black]{$b$} -- (16,2);
\draw[|-|] (14,3) -- (18,3);
\draw[|-|, thick, Purple] (13,4) -- (15,4) node[below, black]{$e$} -- (17,4);
\draw[very thick, |-|] (2,3) -- (4,3) -- (6,3);
\draw[very thick,|-|] (2.5,4) -- (4.5,4) node[below, black]{$z\in\{e,f\}$} -- (6.5,4);

\draw[dashed] (0,-1) -- (0,5);
\draw[dashed] (18.5,-1) -- (18.5,5);
    
\end{tikzpicture}}};
        \node (abcadbee) at (-4,-15) {$z=e$};
        \node (abcadbef) at (4,-15) {$z=f$};
        \node (vd1) at ($(abcadbee)+(0,-0.85)$) {$\vdots$};
        \node (vd2) at ($(abcadbef)+(0,-0.85)$) {$\vdots$};

        \path [->] (abcad) edge (abcadx);
        \path [-] (abcadx) edge (abcade);
        \path [-] (abcadx) edge (abcadc);
        \path [-] (abcadx) edge (abcadb);
        \path [->] (abcadc) edge (blank1);
        \path [->] (abcadb) edge (abcadbx);
        \path [-] (abcadbx) edge (abcadbc);
        \path [->] (abcadbc) edge (blank2);
        \path [->] (blank2) edge (ed3);
        \path [->] (abcade) edge (ed1);
        \path [->] (blank1) edge (ed2);
        \path [->] (blank1) edge (bd2);
        \path [-] (abcadbx) edge (abcadbe);
        \path [->] (abcadbe) edge (abcadbex);
        \path [-] (abcadbex) edge (abcadbee);
        \path [-] (abcadbex) edge (abcadbef);
        \path [-] (abcadbee) edge (vd1);
        \path [-] (abcadbef) edge (vd2);
        
    \end{tikzpicture}
    \caption{$\abcad\rightarrow\game$}
    \label{fig:abcad-g}
\end{figure}

\begin{figure}
    \centering
    \begin{tikzpicture}
        \node (abcadbe-e) at (0,0) {$z=e$};
        \node (d1) at ($(abcadbe-e)+(0,1)$) {$\vdots$};
        \path (abcadbe-e) edge (d1);
        \node[draw=black] (g) at (0,-3.5) {\resizebox{4.5in}{!}{\begin{tikzpicture}

\draw[-|, thick, Red] (0,0) -- (2,0) node[below, black]{$a$} -- (4,0);
\draw[|-|, thick, Red] (8.5,0) -- (10.5,0) node[below, black]{$a$} -- (12.5,0);
\draw[|-|, thick, OliveGreen] (3,1) -- (5,1) node[below, black]{$b$} -- (7,1);
\draw[|-|, thick, Orange] (11,1) -- (13,1) node[below, black]{$d$} -- (15,1);
\draw[|-|, thick, Blue] (5,2) -- (7,2) node[below, black]{$c$} -- (9,2);
\draw[|-|, thick, OliveGreen] (12,2) -- (14,2) node[below, black]{$b$} -- (16,2);
\draw[|-|] (14,3) -- (18,3);
\draw[|-|, thick, Purple] (13,4) -- (15,4) node[below, black]{$e$} -- (17,4);
\draw[|-|] (2,3) -- (6,3);
\draw[|-|, thick, Purple] (2.5,4) -- (4.5,4) node[below, black]{$e$} -- (6.5,4);
\draw[very thick, |-|, Cyan] (6.4,5) -- (8.8,5) node[below, black]{$f$} -- (11.2,5);
\draw[very thick, |-|, Magenta] (8.75,6) -- (11.25,6) node[below, black]{$g$} -- (13.75,6);

\draw[dashed] (0,-1) -- (0,7);
\draw[dashed] (18.5,-1) -- (18.5,7);
    
\end{tikzpicture}}};
        \draw[->] (abcadbe-e) edge (g);

        \node (abcadbef) at (6,0) {$z=f$};
        \node (d2) at ($(abcadbef)+(0,1)$) {$\vdots$};
        \path (abcadbef) edge (d2);
        \node[draw=black] (g2) at (0,-12) {\resizebox{4.6in}{!}{\begin{tikzpicture}

\draw[-|, thick, Red] (0,0) -- (2,0) node[below, black]{$a$} -- (4,0);
\draw[|-|, thick, Red] (8.5,0) -- (10.5,0) node[below, black]{$a$} -- (12.5,0);
\draw[|-|, thick, OliveGreen] (3,1) -- (5,1) node[below, black]{$b$} -- (7,1);
\draw[|-|, thick, Orange] (11,1) -- (13,1) node[below, black]{$d$} -- (15,1);
\draw[|-|, thick, Blue] (5,2) -- (7,2) node[below, black]{$c$} -- (9,2);
\draw[|-|, thick, OliveGreen] (12,2) -- (14,2) node[below, black]{$b$} -- (16,2);
\draw[|-|] (14,3) -- (18,3);
\draw[|-|, thick, Purple] (13,4) -- (15,4) node[below, black]{$e$} -- (17,4);
\draw[|-|] (2,3) -- (6,3);
\draw[|-|, thick, Cyan] (2.5,4) -- (4.5,4) node[below, black]{$f$} -- (6.5,4);
\draw[very thick, |-|] (6.75,4) -- (9.25,4) node[below, black]{$w\in\{e,f\}$} -- (11.5,4);
\draw (0,-1) -- (18.5,-1);
\draw[very thick, |-|, Magenta] (6.25,-2) -- (8.75,-2) node[below, black]{$g$} -- (11.25,-2);
\draw (0,-3) -- (18.5,-3);
\draw[very thick, |-|, Magenta] (8.75,-4) -- (11.25,-4) node[below, black]{$g$} -- (13.75,-4);

\node[anchor = north west] at (0,-1) (oh) {If $w=e$:};
\node[anchor = north west] at (0,-3) (oh2) {If $w=f$:};

\draw[dashed] (0,-5) -- (0,5);
\draw[dashed] (18.5,-5) -- (18.5,5);
    
\end{tikzpicture}}};
        \path[-] (abcadbef) edge (6,-7);
        \path[-] (6,-7) edge (0,-8);
        \path[->] (0,-8) edge (g2);

    \end{tikzpicture}
    \caption{$\abcad\rightarrow\game$ (continued.)}
    \label{fig:abcad-g2}
\end{figure}

Assume we start in an $\abcad$ state. We present one interval as in Figure \ref{fig:abcad-g}. Let $x$ denote the color assigned to the interval by \alice. Then, $x\in\{b,c,e\}$. If $x=e$, then the result is an $\ed$ state. If $x=c$, then notice that if we in introduce a new interval between the interval colored $d$ and the interval colored $x=e$, then updating the right wall to exclude only the interval colored $x=e$ results in a either a $\bd$ state or an $\ed$ state. Hence, we may assume that $x=b$. 

We present two intervals as shown in Figure \ref{fig:abcad-g} while separating any interval assigned the color $a$ to the right. Let $y$ denote the color assigned to the left interval by \alice. Then, $y\in\{c,e,f\}$. If $y=c$, then notice that if we introduce a new interval $[l,r]$ with $l$ immediately to the left of the left endpoint of the interval colored $d$ and $r$ immediately to the right of the left endpoint of the interval colored $y=c$, then updating the right wall to be immediately to the right of the left endpoint of the interval colored $d$ results in an $\ed$ state. Hence, we may assume that the interval was colored a new color. Without loss of generality, let us assume that $y=e$. 

We present two intervals as shown in Figure \ref{fig:abcad-g} while separating any interval assigned the color  $d$ to the left. Let $z$ denote the color that was assigned to the right interval by \alice. Then $z\in\{e,f\}$. If $z=e$, then Figure \ref{fig:abcad-g2} shows how to force colors $f$ and $g$ in two moves. Therefore, we may assume that $z=f$. We present an interval as shown in Figure \ref{fig:abcad-g2}. Let $w$ denote the color assigned to the interval by \alice. Then, $w\in\{e,f\}$. Regardless of the choice of color by \alice, we can force the color $g$ in one move. Figure \ref{fig:abcad-g2} shows how to force $g$ in one move for each case. Thus, we have shown that we can force a $\game$ state from any $\abcad$ state.

\subsection{$(\aab\lor\bab\rightarrow\abcax)$}
\label{sec:aborbab}

\begin{figure}
    \centering
    \begin{tikzpicture}
        \node[draw=black] (aab) at (0,-2) {$\aab$};
        \node[draw=black] (aab-xy) at (0,-4) {\resizebox{4in}{!}{\begin{tikzpicture}

\draw[-|, thick, Red] (0,0) -- (2,0) node[below, black]{$a$} -- (4,0);
\draw[|-|, thick, Red] (8,0) -- (10,0) node[below, black]{$a$} -- (12,0);

\draw[|-|, thick, OliveGreen] (11,1) -- (13,1) node[below, black]{$b$} -- (15,1);

\draw[very thick, |-|] (2,1) -- (4,1) node[below, black]{$x$} -- (6,1);
\draw[very thick, |-|] (5,2) -- (7,2) node[below, black]{$y$} -- (9,2);

\draw[dashed] (0,-1) -- (0,3);
\draw[dashed] (15.5,-1) -- (15.5,3);

\end{tikzpicture}}};
        \node (aab-bc) at (-4.5,-6.5) {\scriptsize$x=b,y=c$};
        \node (aab-cb) at (-1.5,-6.5) {\scriptsize$x=c,y=b$};
        \node (aab-cd) at (1.5,-6.5) {\scriptsize$x=c,y=d$};
        \node (aab-dc) at (4.5,-6.5) {\scriptsize$x=d,y=c$};
        \node[draw=black] (abcab) at (-4.5,-8) {$\abcab$};
        \node[draw=black] (abcac) at (-1.5,-8) {$\abcac$};
        \node[draw=black] (abcad) at (3,-8) {$\abcad$};

        \path [->] (aab) edge (aab-xy);
        \path [-] (aab-xy) edge (aab-bc);
        \path [-] (aab-xy) edge (aab-cb);
        \path [-] (aab-xy) edge (aab-cd);
        \path [-] (aab-xy) edge (aab-dc);
        \path [->] (aab-bc) edge (abcab);
        \path [->] (aab-cb) edge (abcac);
        \path [->] (aab-cd) edge (abcad);
        \path [->] (aab-dc) edge (abcad);
    \end{tikzpicture}
    \caption{$\aab\rightarrow\abcax$}
    \label{fig:aab-proof}
\end{figure}

\begin{figure}
    \centering
    \begin{tikzpicture}
        \node[draw=black] (bab) at (0,-2) {$\bab$};
        \node[draw=black] (bab-cx) at (0,-4) {\resizebox{4in}{!}{\begin{tikzpicture}

\draw[-|, thick, OliveGreen] (2,0) -- (4,0) node[below, black]{$b$} -- (6,0);
\draw[|-|, thick, Red] (8,0) -- (10,0) node[below, black]{$a$} -- (12,0);

\draw[|-|, thick, OliveGreen] (11,1) -- (13,1) node[below, black]{$b$} -- (15,1);

\draw[very thick, |-|] (5,1) node[left]{1} -- (7,1) node[below, black]{$c$} -- (9,1);
\draw[very thick, |-|] (14,2) node[left]{2} -- (16,2) node[below, black]{$x\in\{a,c,d\}$} -- (18,2);

\draw[dashed] (2,-1) -- (2,3);
\draw[dashed] (18.5,-1) -- (18.5,3);

\end{tikzpicture}}};
        \node (bab-ca) at (-4,-6) {\scriptsize$x=a$};
        \node (bab-cc) at (0,-6) {\scriptsize$x=c$};
        \node (bab-cd) at (4,-6) {\scriptsize$x=d$};
        \node[draw=black] (abcac) at (-4,-7) {$\abcac$};
        \node[draw=black] (abcab) at (0,-7) {$\abcab$};
        \node[draw=black] (abcad) at (4,-7) {$\abcad$};

        \path [->] (bab) edge (bab-cx);
        \path [-] (bab-cx) edge (bab-ca);
        \path [-] (bab-cx) edge (bab-cc);
        \path [-] (bab-cx) edge (bab-cd);
        \path [->] (bab-ca) edge (abcac);
        \path [->] (bab-cc) edge (abcab);
        \path [->] (bab-cd) edge (abcad);
    \end{tikzpicture}
    \caption{$\bab\rightarrow\abcax$}
    \label{fig:bab-proof}
\end{figure}

Assume we start in an $\aab$ state. Figure \ref{fig:aab-proof} shows how to force an $\abcax$ state from an $\aab$ state. We present two intervals as shown in Figure \ref{fig:aab-proof}. The possibilities for the two colors are $(b,c)$, $(c,b)$, $(c,d)$ and $(d,c)$. If the intervals are colored $b$ and $c$, then the result is an $\abcab$ state, If the intervals are colored $c$ and $b$, then the result is an $\abcac$ state, and if the intervals are colored $c$ and $d$, or $d$ and $c$, then the result is an $\abcad$ state. Hence, we can force an $\abcax$ state from any $\aab$ state. 

Assume we start in a $\bab$ state. Figure \ref{fig:bab-proof} shows how to force an $\abcax$ state from a $\bab$ state. We present two intervals as shown in Figure \ref{fig:bab-proof}. The first interval on the left must be colored $c$. The possible colors for the second interval are $a$, $c$ and $d$. If the interval is colored $a$, then the result is an $\abcac$ state, if the interval is colored $c$, then the result is an $\abcab$ state, and if the interval is colored $d$, then the result is an $\abcad$ state. Hence, we can force an $\abcax$ state from any $\bab$ state. Since we can force $g$ from any $\abcax$ state, we can force a $\game$ state from any $\aab$ state and from any $\bab$ state.

\subsection{$(\abcde\rightarrow\game)$}
\label{sec:abcde}

\begin{figure}
    \centering
    \begin{tikzpicture}
        \node[draw=black] (abcde) at (0,0) {$\abcde$};
        \node[draw=black] (abcde-x) at (0,-3) {\resizebox{4.5in}{!}{\begin{tikzpicture}

\draw[|-|, thick, Red] (0,0) -- (2,0) node[below, black]{$a$} -- (4,0);
\draw[|-, thick, Blue] (6,0) -- (9,0) node[below, black]{$c$} -- (10,0);
\draw[|-|, thick, OliveGreen] (1,1) -- (3,1) node[below, black]{$b$} -- (5,1);
\draw[|-, thick, Orange] (7,1) -- (9,1) node[below, black]{$d$} -- (10,1);
\draw[|-, thick, Purple] (8,2) -- (9,2) node[below, black]{$e$} -- (10,2);
\draw[very thick, |-|] (-6,0) -- (-2,0);
\draw[very thick, |-|] (-5.5,1) -- (-1.5,1);
\draw[very thick, |-|] (-5,2) -- (-1,2);

\draw[very thick, |-|] (-4.5,3) -- (-2.5,3) node[below, black]{$x\in\{a,c,d\}$} -- (-0.5,3);

\draw[dashed] (10,-1) -- (10,4);
\draw[dashed] (-6.5,-1) -- (-6.5,4);
\draw[very thick, dashed] (-0.75,-1) -- (-0.75,4);
\path[thick, ->] (-6.25,-1) edge (-1,-1);

\end{tikzpicture}}};
        \node (abcde-c) at (-4,-6) {$x=c$};
        \node (abcde-d) at (0,-6) {$x=d$};
        \node (abcde-a) at (4,-6) {$x=a$};
        \node[draw=black] (aab) at (5,-9) {$\aab$};
        \node[draw=black] (abcde-x-y) at (-1.8,-9) {\resizebox{3.5in}{!}{\begin{tikzpicture}

\draw[|-|, thick, Red] (0,0) -- (2,0) node[below, black]{$a$} -- (4,0);
\draw[|-, thick, Blue] (6,0) -- (9,0) node[below, black]{$c$} -- (10,0);
\draw[|-|, thick, OliveGreen] (1,1) -- (3,1) node[below, black]{$b$} -- (5,1);
\draw[|-, thick, Orange] (7,1) -- (9,1) node[below, black]{$d$} -- (10,1);
\draw[|-, thick, Purple] (8,2) -- (9,2) node[below, black]{$e$} -- (10,2);

\draw[-|, thick, Blue] (-2.5,1) -- (-1.5,1) node[below, black]{$x\in\{c,d\}$} -- (-0.5,1);
\draw[very thick, |-|] (-2,2) -- (2,2);
\draw[very thick, |-|] (-1,3) -- (1,3) node[below, black]{$y\in\{e,f\}$} -- (3,3);

\draw[dashed] (10,-1) -- (10,4);
\draw[dashed] (-2.5,-1) -- (-2.5,4);
\draw[very thick, dashed] (2.5,-1) -- (2.5,4);
\path[thick, ->] (-2.25,-1) edge (2.25,-1);

\end{tikzpicture}}};
        \node (hi) at (-3,-12) {$y=e$};
        \node[draw=black] (abcde-x-e) at (-3.2,-15.5) {\resizebox{2.25in}{!}{\begin{tikzpicture}

\draw[-|, thick, Red] (2,0) -- (3,0) node[below, black]{$a$} -- (4,0);
\draw[|-, thick, Blue] (6,0) -- (9,0) node[below, black]{$c$} -- (10,0);
\draw[-|, thick, OliveGreen] (2,1) -- (3,1) node[below, black]{$b$} -- (5,1);
\draw[|-, thick, Orange] (7,1) -- (9,1) node[below, black]{$d$} -- (10,1);
\draw[|-, thick, Purple] (8,2) -- (9,2) node[below, black]{$e$} -- (10,2);

\draw[-|, thick, Purple] (2,2) -- (3,2) node[below, black]{$e$} -- (3.5,2);

\draw[very thick, |-|] (3,3) -- (5.25,3) node[below, black]{$f$} -- (7.5,3);
\draw[very thick, |-|] (3.75,4) -- (6.75,4) node[below, black]{$g$} -- (9.75,4);

\draw[dashed] (10,-1) -- (10,5);
\draw[dashed] (2,-1) -- (2,5);

\end{tikzpicture}}};
        \node (hi2) at (3,-12) {$y=f$};
        \node[draw=black] (abcde-x-f) at (3.2,-15.5) {\resizebox{2.25in}{!}{\begin{tikzpicture}

\draw[-|, thick, Red] (2,0) -- (3,0) node[below, black]{$a$} -- (4,0);
\draw[|-, thick, Blue] (6,0) -- (9,0) node[below, black]{$c$} -- (10,0);
\draw[-|, thick, OliveGreen] (2,1) -- (3,1) node[below, black]{$b$} -- (5,1);
\draw[|-, thick, Orange] (7,1) -- (9,1) node[below, black]{$d$} -- (10,1);
\draw[|-, thick, Purple] (8,2) -- (9,2) node[below, black]{$e$} -- (10,2);
%\draw[-|, thick, OliveGreen] (-2,2) -- (0,2) node[below, black]{$y$} -- (2,2);
\draw[-|, thick, Cyan] (2,2) -- (3,2) node[below, black]{$f$} -- (3.5,2);

\draw[very thick, |-|] (3,3) -- (6,3) node[below, black]{$g$} -- (9,3);

\draw[dashed] (10,-1) -- (10,4);
\draw[dashed] (2,-1) -- (2,4);

\end{tikzpicture}}};

        \path [->] (abcde) edge (abcde-x);
        \path [-] (abcde-x) edge (abcde-a);
        \path [->] (abcde-a) edge (aab);
        \path [-] (abcde-x) edge (abcde-d);
        \path [-] (abcde-x) edge (abcde-c);
        \path [->] (abcde-d) edge (abcde-x-y);
        \path [->] (abcde-c) edge (abcde-x-y);
        \path [-] (abcde-x-y) edge (hi);
        \path [->] (hi) edge (abcde-x-e);
        \path [-] (abcde-x-y) edge (hi2);
        \path [->] (hi2) edge (abcde-x-f);

    \end{tikzpicture}
    \caption{$\abcde\rightarrow\aab\lor\game$}
    \label{fig:abcde-proof}
\end{figure}

Assume we start in an $\abcde$ state. Figure \ref{fig:abcde-proof} shows how to force a $\game$ state in three phases. In the first phase, we present four intervals while separating any interval assigned a color in $\{b,e,f\}$ to the left. The right-most interval must be colored either $a$, $c$ or $d$. If the interval is colored $a$, then the result is an $\aab$ state. Hence, we may assume that the interval was colored $c$ or $d$. In the second phase, we present two intervals while separating any interval assigned a color in $\{e,f\}$ to the right. It is easy to see that in order to avoid using $g$, the interval on the right must be colored either $e$ or $f$. The bottom two frames of Figure \ref{fig:abcde-proof} show how to force a $\game$ state for each of the two cases.

\subsection{$(\game)$}
\label{sec:gg}

\begin{figure}[h]
    \centering
    \begin{tikzpicture}

\draw[|-|, thick, OliveGreen] (-2,1) node[left,black]{$2$} -- (0,1) node[below, black]{$b$} -- (2,1);
\draw[|-|, thick, Red] (-1,0) node[left,black]{$1$} -- (1,0) node[below, black]{$a$} -- (3,0);
\draw[|-|, thick, Blue] (7,0) node[left,black]{$3$} -- (9,0) node[below, black]{$c$} -- (11,0);
\draw[|-|, thick, Orange] (6,1) node[left,black]{$4$} -- (8,1) node[below, black]{$d$} -- (10,1);
\draw[|-|, thick, Purple] (5,2) node[left,black]{$5$} -- (7,2) node[below, black]{$e$} -- (9,2);

\end{tikzpicture}
    \caption{$\aab^*\lor\bab^*\lor\abcde$}
    \label{fig:gg}
\end{figure}

Finally, we show that we can always force an $\aab^*$ state, a $\bab^*$ state or an $\abcde$ state. To show this, we simply introduce intervals in the order shown in Figure \ref{fig:gg}. Intervals 1 and 2 must be assigned colors $a$ and $b$, respectively. When presented, intervals 3, 4 and 5 cannot be assigned the color $a$ since the result would be an $\aab^*$ state, nor the color $b$ since the result would be a $\bab^*$ state (which can be seen by updating the right wall to be immediately to the right of the left endpoint of the newest interval). Thus, the result after presenting all five intervals is an $\abcde$ state. This concludes the proof.

%\nocite{*}
\bibliographystyle{abbrvnat}
\bibliography{semi}
\label{sec:biblio}

\end{document}